\documentclass[11pt,a4paper]{amsart}
\usepackage{graphicx}
\usepackage{geometry}
\usepackage{xargs}
\usepackage[active]{srcltx}

\usepackage{enumitem} %label option for enumerate \alph* \Alph* \roman* \Roman* or other... [label=$\bullet$], [label=\alph*]

\usepackage[utf8]{inputenc}
\usepackage[english]{babel}
\usepackage{latexsym}
\usepackage{amssymb}
\usepackage{upgreek}
\usepackage{bbm}
\usepackage[norelsize,ruled,vlined]{algorithm2e}
\usepackage{amsmath}
\usepackage[textwidth=4cm,textsize=footnotesize]{todonotes}
\usepackage[colorlinks,linkcolor=blue,citecolor=red,urlcolor=blue]{hyperref}
\usepackage{accents}
\usepackage{dsfont}
\usepackage{aliascnt}
\usepackage{cleveref}
\makeatletter
\newtheorem{theorem}{Theorem}
\crefname{theorem}{theorem}{Theorems}
\Crefname{Theorem}{Theorem}{Theorems}

\newaliascnt{lemma}{theorem}
\newtheorem{lemma}[lemma]{Lemma}
\aliascntresetthe{lemma}
\crefname{lemma}{lemma}{lemmas}
\Crefname{Lemma}{Lemma}{Lemmas}

\newaliascnt{corollary}{theorem}
\newtheorem{corollary}[corollary]{Corollary}
\aliascntresetthe{corollary}
\crefname{corollary}{corollary}{corollaries}
\Crefname{Corollary}{Corollary}{Corollaries}

\newaliascnt{proposition}{theorem}

\aliascntresetthe{proposition}
\crefname{proposition}{proposition}{propositions}
\Crefname{Proposition}{Proposition}{Propositions}

\newaliascnt{definition}{theorem}

\aliascntresetthe{definition}
\crefname{definition}{definition}{definitions}
\Crefname{Definition}{Definition}{Definitions}

\newaliascnt{remark}{theorem}
\newtheorem{remark}[remark]{Remark}
\aliascntresetthe{remark}
\crefname{remark}{remark}{remarks}
\Crefname{Remark}{Remark}{Remarks}

\newtheorem{example}[theorem]{Example}
\crefname{example}{example}{examples}
\Crefname{Example}{Example}{Examples}

\crefname{figure}{figure}{figures}
\Crefname{Figure}{Figure}{Figures}

\newtheorem{assumption}{\textbf{H}\hspace{-3pt}}
\Crefname{assumption}{\textbf{H}\hspace{-3pt}}{\textbf{H}\hspace{-3pt}}
\crefname{assumption}{\textbf{H}}{\textbf{H}}

\Crefname{assumptionL}{\textbf{L}\hspace{-3pt}}{\textbf{L}\hspace{-3pt}}
\crefname{assumptionL}{\textbf{L}}{\textbf{L}}

\Crefname{assumptionG}{\textbf{G}\hspace{-3pt}}{\textbf{G}\hspace{-3pt}}
\crefname{assumptionG}{\textbf{G}}{\textbf{G}}

\definecolor{red}{rgb}{1,0,0}
\definecolor{green}{rgb}{0,1,0}
\definecolor{blue}{rgb}{0,0,1}

\title{{An elementary approach to uniform in time propagation of chaos}}

 \author[]{Alain Durmus, Andreas Eberle, Arnaud Guillin, Raphael Zimmer}

\address{ENS Paris-Saclay\\
Centre de Math\'ematiques et leurs Applications\\
61 avenue du Pr\'esident Wilson\\
94235 Cachan, France}
\email{alain.durmus@cmla.ens-cachan.fr}

\address{Universit\"at Bonn\\
Institut f\"ur Angewandte Mathematik\\
 Endenicher Allee 60\\
  53115 Bonn, Germany}
\email{eberle@uni-bonn.de, Raphael@infoZimmer.de}
\urladdr{http://www.uni-bonn.de/$\sim$eberle}

\address{Laboratoire de Math\'ematiques Blaise Pascal\\
CNRS - UMR 6620\\
Universit\'e Clermont-Auvergne\\
Avenue des landais,\\
63177 Aubiere cedex, France}
\email{guillin@math.univ-bpclermont.fr}
\urladdr{http://math.univ-bpclermont.fr/$\sim$guillin}

\date{}

\newcommandx{\functionspace}[2][1=+]{\mathbb{F}_{#1}(#2)}

\newcommand{\2}[2]{\2{#1}{#2}}

\newcommand{\borelSet}{\mathcal{B}}

\newcommand{\LeftEqNo}{\let\veqno\@@leqno}

\newcommand{\PE}{\mathbb{E}}

\newcommandx{\Vnorm}[2][1=V]{\| #2 \|_{#1}}
\newcommandx{\VnormEq}[2][1=V]{\left\| #2 \right\|_{#1}}
\newcommandx{\norm}[2][1=]{\ifthenelse{\equal{#1}{}}{\left\Vert #2 \right\Vert}{\left\Vert #2 \right\Vert^{#1}}}
\newcommandx{\normsup}[2][1=]{\ifthenelse{\equal{#1}{}}{\left\Vert #2 \right\Vert_{\infty}}{\left\Vert #2 \right\Vert^{#1}_{\infty}}}
\newcommandx{\normLigne}[2][1=]{\ifthenelse{\equal{#1}{}}{\Vert #2 \Vert}{\Vert #2\Vert^{#1}}}
\newcommandx{\normLine}[2][1=]{\ifthenelse{\equal{#1}{}}{\Vert #2 \Vert}{\Vert #2\Vert^{#1}}}

\newcommand{\parenthese}[1]{\left(#1 \right)}

\newcommand{\defEns}[1]{\left\lbrace #1 \right\rbrace }

\newcommand{\ps}[2]{\left\langle#1,#2 \right\rangle}

\newcommandx\probaMarkovTilde[2][2=]
{\ifthenelse{\equal{#2}{}}{{\widetilde{\mathbb{P}}_{#1}}}{\widetilde{\mathbb{P}}_{#1}\left[ #2\right]}}

\newcommandx{\expe}[2][1=]{\ifthenelse{\equal{#1}{}}{\PE \left[ #2 \right]}{\PE^{#1}\left[ #2 \right]}}

\def\eqsp{\;}
\newcommand{\coint}[1]{\left[#1\right)}

\newcommand{\ooint}[1]{\left(#1\right)}
\newcommand{\ccint}[1]{\left[#1\right]}

\newcommandx{\weight}[2][2=n]{\omega_{#1,#2}^N}

\def\rmd{\mathrm{d}}
\newcommandx\sequence[3][2=,3=]
{\ifthenelse{\equal{#3}{}}{\ensuremath{( #1_{#2})}}{\ensuremath{( #1_{#2})_{ #2 \in #3}}}}
\newcommandx{\sequencen}[2][2=n\in\N]{\ensuremath{(#1)_{#2}}}
\newcommandx\sequenceDouble[4][3=,4=]
{\ifthenelse{\equal{#3}{}}{\ensuremath{\{ (#1_{#3},#2_{#3}) \}}}{\ensuremath{\{  (#1_{#3},#2_{#3}), \eqsp #3 \in #4 \}}}}
\newcommandx{\sequencenDouble}[3][3=n\in\N]{\ensuremath{\{ (#1_{n},#2_{n}), \eqsp #3 \}}}

\def\rme{\mathrm{e}}

\def\rset{\mathbb{R}}

\def\nset{\mathbb{N}}

\newcommandx{\CPE}[3][1=]{{\mathbb E}^{#3}_{#1}\left[#2 \right]} %%%% esperance conditionnelle
\newcommandx{\CPVar}[3][1=]{\mathrm{Var}^{#3}_{#1}\left\{ #2 \right\}}
\newcommand{\CPP}[3][]
{\ifthenelse{\equal{#1}{}}{{\mathbb P}\left(\left. #2 \, \right| #3 \right)}{{\mathbb P}_{#1}\left(\left. #2 \, \right | #3 \right)}}

\newcommand{\chunk}[4][]%
{\ifthenelse{\equal{#1}{}}{\ensuremath{{#2}_{#3:#4}}}{\ensuremath{#2^#1}_{#3:#4}}
}

\def\Id{\operatorname{Id}}

\def\eventA{\mathsf{A}}

\newcommand{\N}{\mathbb{N}}

\newcommandx{\Phiverlet}[2][1=,2=]{\ifthenelse{\equal{#1}{}}{\Phi}{\Phi_{#1}^{\circ (#2)}}}

\def\barX{\bar{X}}
\newcommandx{\lawX}[2][2=]{\mathcal{L}(#1_{#2})}

\def\phir{\phi_{\operatorname{r}}}
\def\phis{\phi_{\operatorname{s}}}

\def\tildeB{\tilde{B}}
\def\diffX{E}
\def\diffXn{e}

\def\transpose{\operatorname{T}}
\def\RZ{R_0}
\def\RO{R_1}
\def\f{f}
\def\g{g}
\def\c{c}
\newcommandx{\fc}[1][1=]{\ifthenelse{\equal{#1}{}}{\f}{\f\left( #1\right)}}
\newcommandx{\fcp}[1][1=]{\ifthenelse{\equal{#1}{}}{\f'}{\f'\left( #1\right)}}
\newcommandx{\fcpp}[1][1=]{\ifthenelse{\equal{#1}{}}{\f''}{\f''\left( #1\right)}}

\def\wasserstein{\mathcal{W}}
\def\ctheo{C}

\def\mV{m_V}
\def\MV{M_V}

\def\barmu{\bar{\mu}}
\def\etal{et al.}

\begin{document}

\begin{abstract}
Based on a coupling approach,
we prove uniform in time propagation of chaos for weakly interacting mean-field particle systems with possibly non-convex confinement and interaction potentials. The approach is based on a combination of reflection and synchronous couplings applied to the individual particles. It provides explicit quantitative bounds that significantly extend previous results for the convex case.
%
%the granular media equations in the non convex potentials case, generalizing results of \cite{malrieu03,cgm-08,DMT15}, valid only for uniformly strictly or degenerately strictly convex potentials. As a by-product, we obtain the exponential convergence to the equilibrium of the particles system, uniform in the number of particles, under similar assumptions on the potentials than for the convergence of the non-linear equation.
\end{abstract}
\maketitle
\section{Introduction}

Let $V,W : \rset^d \to \rset$ be twice continuously differentiable functions satisfying appropriate regularity and growth conditions. We
consider the \emph{kinetic Fokker-Planck equation}
\begin{equation}
\label{eq:non_linear_pde}
\partial_t \mu_t=\nabla\cdot[\nabla \mu_t+(\nabla V+\nabla W *\mu_t)\mu_t]
\end{equation}
and its probabilistic counterpart, the nonlinear stochastic differential equation 
\begin{equation}
  \label{eq:non_linear_sde}
\rmd \barX_t \ =\ -\nabla V(\barX_t) \, \rmd t\, -\, \nabla W \ast \mu_t (\barX_t)\,  \rmd t\,  + \, \sqrt{2}\, \rmd B_t \eqsp, \qquad
\mu_t\ =\  \lawX{\bar X}[t] \eqsp,
\end{equation}
of \emph{McKean-Vlasov} type.
Here $\mu_t$ is a time dependent probability measure on $\mathbb R^d$ and $*$ denotes the standard convolution operator. The function $V$ corresponds to a \emph{confinement potential} and the function $W$ to an \emph{interaction potential}.
Variants of the equation occur for example in the modelling of granular media, cf.\ %Benedetto \etal~
\cite{BCCP,VillaniGranular}.
%Let us explain quickly the origin of this equation used in the modelling of granular media, as described (in a particular case) by Benedetto \etal~\cite{BCCP}. Many infinitesimal particles are colliding inestically, and by a correct renormalization between the inelasticity and the frequency of the collisions, $\mu_t$ may be seen as the velocity of a representative particle (among an infinity). The potential $V$ represents the friction, also called confinement potential, and $W$ the inelastic collisions between particles with different velocities, also called interaction potential.

Both in the probability and in the p.d.e.\ community, existence and uniqueness 
of \eqref{eq:non_linear_pde} and \eqref{eq:non_linear_sde}
have attracted much attention, see \cite{McK,Fun,Szn91,Mel96} for a few milestones, and \cite{MV} and \cite{HSS} for two recent results.  During the last twenty years, there has been a lot
of progress on convergence to equilibrium of solutions $(\mu_t)_{t \geq 0}$ of
\eqref{eq:non_linear_pde}. Carrillo, McCann and Villani
\cite{cmcv-03,cmcv-06} have proven an exponential convergence rate
under the strict convexity condition
$\mbox{Hess}(V+2W)\ge\rho \Id$ with $\rho >0$. They have also established a
polynomial convergence rate in the case where $V+2W$ is only
degenerately strictly convex with $\mathrm{Hess}(V+2W)(x) = 0$ for
some isolated points $x \in \rset^d$. Malrieu
\cite{malrieu03} and Cattiaux \etal~\cite{cgm-08}
have developed a probabilistic approach to these results that is based 
on an approximation by the mean-field particle system which is
defined for $N\in\mathbb N$ as the solution of the equations
\begin{equation}
  \label{eq:sde_particle}
\rmd X_t^{i,N} = -\nabla V(X_t^{i,N})\, \rmd t \, -\, N^{-1} \sum_{j=1}^N \nabla W (X_t^{i,N}-X_t^{j,N})\, \rmd t \,  +\, \sqrt{2} \rmd B_t^{i} \eqsp,\quad i=1,\ldots ,N,
% \begin{cases}
% \rmd X_s^{i,N} &= -\nabla V(X_s^{i,N}) \rmd s - N^{-1} \sum_{j=1}^N \nabla W (X_s^{i,N}-X_s^{j,N}) \rmd s  + \sqrt{2} \rmd B_s^{i}  \\
% X_0^{i,N} &= X_0^{i} \eqsp,
% \end{cases}
\end{equation}
where the initial values $X_0^1,\ldots ,X_0^N$ are i.i.d.\ random variables, and the processes $(B_t^1)_{t\ge 0}$, $\ldots ,(B_t^N)_{t\ge 0}$ are independent Brownian motions.  It is well-known
\cite{Szn91,Mel96} that under weak assumptions on $V$ and $W$, the laws of the
particles at time $t$ converge to the solution $\mu_t$ of \eqref{eq:non_linear_pde} as $N\to\infty$. In \cite{malrieu03,cgm-08},
%for $ i \in \{1,\cdots,N\}$, $N \in \nset$, $N \geq 1$, and $(X_0)_{1\le i\le N}$ are i.i.d. random variables, when $V$ and $W$ no more satisfy uniform convexity hypotheses. %, where for all $i \in \{1,\cdots,N\}$, $ X_0^{i}$ is $\rset^d$-valued random variable. 
both convergence to equilibrium for the nonlinear SDE \eqref{eq:non_linear_sde} and uniform in time propagation of chaos for the particle system have been proven under convexity assumptions
by using functional inequalities and synchronous couplings, respectively.\smallskip

When $V+2W$ is not convex, the situation is much more delicate.
Uniqueness of a stationary solution of \eqref{eq:non_linear_pde} does
not hold in general without additional conditions on $V$ and $W$. Even in this case,
only few results on convergence to equilibrium are known. By a direct study
of the dissipation of the Wasserstein distance, Bolley, Gentil and Guillin
\cite{bgg-13} established an exponential trend to equilibrium in
a weakly non-convex case. In a recent work by three of the authors \cite{egz-16},
an exponential contraction property and thus convergence to equilibrium could be
shown for a much broader class of potentials. The proof is based on a new coupling approach originating in \cite{Eberle2015}, which is also the basis of this work.
An interesting related problem 
arises when there are multiple invariant measures, see
\cite{tugaut2012}. In this widely open case, a main interest are the relative basins
of attraction of the equilibria.\smallskip

%The second major issue concerns the approximation
%of~\eqref{eq:non_linear_pde} by a linear system of particles. To do so
%let us rewrite \eqref{eq:non_linear_pde} by a stochastic differential
%equation (SDE) 
%
%where $\Leb^d$ is the Lebesgue measure on $\rset^d$, and for all $t
%\geq 0$, $\lawX{X}[t]$ is the distribution of $X_t$.  The main purpose
%of this paper is to provide quantitative estimates between the
%solution in time of \eqref{eq:non_linear_sde} and its particle system approximation 
%
%Once the problem of existence and uniqueness of solutions is settled,
%two main issues are raised concerning the solution of this granular
%media equation: 1) the convergence to equilibrium of
%$(\mu_t)_{t \geq 0}$ as $t \to +\infty$ 2) the approximation of $\mu_t$
%for all $t \geq 0$ by the solution a linear system of particles. 

The convergence of the empirical measure of the particle system \eqref{eq:sde_particle} to the solution of \eqref{eq:non_linear_pde}, or, equivalently, the convergence of the pair empirical measure to the tensor product of two solutions of \eqref{eq:non_linear_pde}, has been stated under the name of \emph{propagation of chaos} by Kac \cite{Kac}, and further developed by Sznitman \cite{Szn91}. However, the corresponding general results are only valid uniformly for a 
fixed time horizon. A 
crucial point is then not only to assert the convergence of the empirical measures but to quantify it. A remarkable analytic framework providing also quantitative propagation of chaos estimates has been
developed by Mischler, Mouhot and their coauthors in connection with Kac's program
in kinetic theory \cite{mm13,hm14,Mmmw15}.\smallskip

In this article, we propose a very different and much more elementary, probabilistic approach to quantitative bounds for
propagation of chaos.
Our main result is a \emph{uniform in time propagation of chaos} bound
that takes the form
 \begin{equation*}
    \wasserstein_{\ell^1 (\fc )}\left( \mathcal L (X_t^{1,N},\ldots, X_t^{N,N}),
    \mu_{t}^{\otimes N}  \right)    
   \  \leq\  A\cdot N^{-1/2}\qquad\text{for all }t\ge 0\text{ and }n\in \mathbb N \eqsp,
  \end{equation*} 
if $X_0^1,\ldots ,X_0^n$ are i.i.d.\ with initial law $\mu_0$, see Theorem \ref{theo:unif_prop} below. Here $ \wasserstein_{\ell^1 (\fc )}$ is an $L^1$ Wasserstein distance on $(\mathbb R^d)^N$ and $A$ is an explicit finite constant, see below for the precise definitions.
The bound holds under similar assumptions as the quantitative bounds on convergence 
to equilibrium for McKean-Vlasov equations in \cite{egz-16}. In particular, it applies in non-convex cases provided the confinement potential $V$ is strictly convex outside a 
ball, and the interaction potential $W$ is symmetric and globally Lip\-schitz continuous with sufficiently small Lipschitz constant. Consequently, our results are significant extensions 
of the uniform in time propagation of chaos results in the convex case in \cite{malrieu03,cgm-08}. The main difference to the approach in these works is that we 
use a more refined coupling which is harder to analyse but much more powerful than synchronous coupling. The second main ingredient in our proofs is an
adequately constructed $L^1$ Wasserstein distance that is well adapted to the couplings we consider. Our probabilistic approach is very different from the 
analytic methods developed in \cite{mm13,hm14,Mmmw15}, and the conditions required are not easily comparable.
One advantage of the coupling approach presented here is that it is very simple
and quite robust. This might facilitate the application to other classes of models. For example, the same argument can be applied immediately if $\nabla V$ is replaced by a non-gradient drift $\beta$ that satisfies corresponding assumptions. 
Similarly, $\nabla W$ can be replaced by a more general interaction term, see Remark \ref{rem:extensions} below.\smallskip

%Namely we want to prove that for all $t \geq 0$ and $N \in \nset$, $N \geq 1$,
%\begin{equation}
%  \label{eq:prop_chaos_dep_t}
%  \dist\left(\lawX{\bar X}[t],\lawX{X^{1,N}}[t]\right)\le \psi(t)\ell(N)
%\end{equation}
%for some distance $\dist$ on the space of probability measures on
%$\rset^d$ and some function $\psi : \rset_+ \to \rset_+$,
%$\ell :\nset \to \rset_+$, with $\ell$
%decreasing. \eqref{eq:prop_chaos_dep_t} have been established for
%functions $\psi$ growing exponentially in time. The objective of this
%paper is to show that under appropriate conditions,
%\eqref{eq:prop_chaos_dep_t} is satisfied for $\psi(t) = 1$ for all
%$t \geq 0$. This type of result is referred to as uniform propagation
%of chaos. Concerning the distance $\dist$, there are many notions
%reinforcing the notions of propagation of chaos : entropic chaos,
%Fisher information chaos... We will be here particularly interested in
%uniform in time propagation of chaos in Wasserstein distance, and will
%try to assert it under similar assumptions than in the case where the
%nonlinear equation converges exponentially fast to its unique
%equilibrium. In this direction, results were first given by
%probabilistic ideas, using synchronous coupling in
%\cite{malrieu03,cgm-08} in the uniformly strictly convex case, or in
%the degenerately convex case under a center of mass assumption (when
%$V$ may be equal to $0$).

%See also the very recent work  \cite{DMT15}, where
%this last condition can be removed under an uniformly
%strictly convex assumption. 

In Section \ref{sec:results}, we present our main hypotheses and
results concerning the uniform in time propagation of chaos. The proofs are provided in 
Section \ref{sec:proofs}. Our approach is based on an interplay between reflection and synchronous coupling. It relies
heavily on the framework introduced in \cite{Eberle2015, egz-16}. Our
results use properties of the confinement
potential, which may nevertheless possess many wells. The interactions
are mainly seen as a perturbation. An interesting and
challenging problem to be taken up in forthcoming work is to prove propagation of chaos in situations where $V=0$ and $W$
is not convex, but uniqueness of a stationary distribution holds.
% and a fixed center of mass.
%\textcolor{red}{Section 1 substantially revised. Please check.}

\section{Uniform in time propagation of chaos}\label{sec:results}

In this section we state our main results. We will first state the precise assumptions on the potentials $V$ and $W$.  Moreover, we define some functions and parameters that will determine the particular Wasserstein distance that we consider.

\subsection{Hypotheses and definitions}

We first state our assumption on the confinement potential.
\begin{assumption}
  \label{assum:kappa}
There  is a continuous function $\kappa : [0,\infty )\to \mathbb R$ satisfying 
$ \liminf\limits_{r \to \infty} \kappa(r) >0$ 
%and $ \lim_{r \to 0} r \kappa(r) = 0$ 
such that
\begin{equation}
\label{eq:def_kappa}
\ps{\nabla V(x)-\nabla V(y)}{x-y} \ \le\ \kappa (\| x-y\| )\,\| x-y\|^2\qquad\text{for all }x,y\in\mathbb R^d \eqsp.
\end{equation}
\end{assumption}
Note that under \Cref{assum:kappa}, there exist $\mV >0$  and $\MV \geq 0$ such that for all $x,y \in \rset^d$,
\begin{equation}
\label{eq:v_convex_inf}
  \ps{\nabla V(x)-\nabla V(y)}{x-y} \geq \mV \norm[2]{x-y} -\MV \eqsp. 
\end{equation}
Following the framework introduced in \cite{Eberle2015}, we now define constants $\RZ$ and $\RO$ as 
\begin{align}
\label{eq:def_RZ}
  \RZ & \ :=\ \inf\defEns{s \in \rset_+ \, : \, \kappa(r) \geq 0 \text{ for all } r \geq s},\\
\label{eq:def_RO}
\RO & \ :=\ \inf \defEns{s \geq \RZ \, : \, s(s-\RZ) \kappa(r) \geq 8 \text{ for all } r \geq s }\eqsp ,
\end{align}
and we consider the functions $\varphi ,\Phi ,g : \rset_+ \to \rset_+$ defined by
\begin{align}
  \label{eq:def_phi_concave}
&\varphi(r) = \exp \parenthese{-\frac 14 \int_0^r s \, \kappa_{-}(s) \rmd s} \eqsp, \quad \Phi(r) = \int_0^r \varphi(s) \rmd s \eqsp,\\
\label{eq:def_g_concave}
&\g(r)  = 1-\frac c2 \int_0^{r \wedge \RO} \Phi(s) \varphi^{-1}(s) \rmd s \eqsp,
\end{align}
where $\kappa_-=\max(0,-\kappa)$, and
\begin{equation}
  \label{eq:def_cont_rate}
  \c\ =\ \left. 1\middle/ \int_0^{\RO} \Phi(s) \varphi^{-1}(s) \rmd s \right. \eqsp.
\end{equation}
Note that $\varphi(r) = \varphi(\RZ)$ for $r \geq \RZ$, and $\g(r) = 1/2$ for
$r \geq \RO$. In addition, for all $r \in \rset_+$,
$\g(r) \in \ccint{1/2,1}$. We now define an increasing function $f:[0,\infty )\to [0,\infty )$
%$\f : \rset_+ \to \rset_+$ 
by
\begin{equation}
  \label{eq:def_f_concave}
\f(r) = \int_0^r \varphi(s) \g(s) \rmd s \eqsp.
\end{equation}
Since $\varphi$ and $g$ are decreasing,  $f$ is concave, {for all }$r\ge 0$,
%function satisfying $f(0) = 0$, for all $r
%\in \rset_+$, $\Phi(r)/2 \leq f(r) \leq\ \Phi(r)\ \le r$ and $\varphi(\RZ)/2 \leq
%\f'(r) \leq 1$. Therefore,
\begin{equation}\label{flinear}
 \varphi(\RZ) r /2\ \le\ \Phi (r)/2 \ \le\ \f(r)\ \le \ \Phi (r)\ \le\ r \, \eqsp.
\end{equation}
Note that by \eqref{flinear}, $(x,y) \mapsto f(\normLigne{x-y})$
induces a distance that is equivalent to the Euclidean distance on $\mathbb R^d$. Below, we will use contraction properties in $L^1$ Wasserstein distances based 
on the underlying distance $f(\| x-y\| )$. These contraction properties will be a consequence of
the inequality
\begin{equation}
\label{eq:iode_f_concave}
  \f''(r) -r \kappa(r) \f'(r)/4\ \leq\ -\c \f(r)/2 \qquad\text{for all }r \in \rset_+ \setminus \{\RO \}\eqsp.
\end{equation}
Indeed, notice that \eqref{eq:iode_f_concave} is satisfied for
$r \in \coint{0,\RO}$ by the definitions and since $f \leq \Phi$. Moreover, for $r >R_1$, $\f(r) = \f(\RO) + \varphi(\RZ)(r - \RO)/2$, and thus by definition of $R_1$,
\begin{multline*}
    \f''(r) -r \kappa(r) \f'(r)/4  = -r \kappa(r)\varphi(\RZ) /8 \leq  -r(\RO(\RO-\RZ))^{-1} \varphi(\RZ)  \\
 \leq -\Phi(r)(\Phi(\RO)(\RO-\RZ))^{-1} \varphi(\RZ)  \leq -\frac 12\Phi(r)
\left/ \int_{\RZ}^{\RO} \Phi(s) \varphi^{-1}(s) \rmd s\right. \leq -\c \f(r)/2 \eqsp.
\end{multline*}
%and for all $r \in \ooint{\RZ,\plusinfty}$ $\Phi(r) = \f(\RZ) +
%(\varphi(\RZ)/2)(r - \RZ)$.  Therefore using that $\varphi$ is constant on
%$\coint{\RZ,\plusinfty}$, $\Phi$ is concave and $f \leq \Phi$, for all $r \in
%\ooint{\RO,\plusinfty}$, we get
Here we have used that for $r\ge R_0$, $\varphi$ is constant, $\Phi(r)=\Phi(R_0)+(r-R_0)\varphi(R_0)$,
and
\begin{eqnarray*}
\lefteqn{\int\nolimits_{R_0}^{R_1}\Phi(s)\varphi(s)^{-1}\:ds 
\ =\ {\Phi(R_0)}{\varphi(R_0)^{-1}}(R_1-R_0)+(R_1-R_0)^2/2}\\
&\geq &(R_1-R_0)\left(\Phi(R_0)+(R_1-R_0)\varphi(R_0)\right)\varphi(R_0)^{-1}/2\
= \ (R_1-R_0)\Phi(R_1)\varphi(R_0)^{-1}/2\eqsp.
\end{eqnarray*}
%secondly \eqref{eq:iode_f_concave} will be used to enable us to show the convergence to equilibrium. \alain{I would suppress this comment}
% \begin{equation}\label{fdiffusion}
% \text{for all }  r\in \ooint{\RO,\plusinfty},\qquad f''(r)-\frac{r\kappa(r)}{4}f'(r)\le -cf(r).
% \end{equation}

Next, we state the assumptions on the interaction potential. 
\begin{assumption}
  \label{assum:W}
  \begin{enumerate}[label=(\roman*)]
  \item \label{assum:W_i} $W$ is symmetric, i.e., $W(x) = W(-x)$ for all $x \in \rset^d$. 
  \item \label{assum:W_ii}  There exists $\eta \in \ooint{0,\c}$ such that for all $x,y \in \rset^d$,
  \begin{equation*}
    \norm{\nabla W(x) - \nabla W(y) } \leq  \eta\, \fc[\norm{x-y}] \eqsp,
  \end{equation*}
  where $\fc$ and $\c$ are defined by \eqref{eq:def_f_concave} and 
  \eqref{eq:def_cont_rate}.
%\item There exists $mW \geq 0$ and $\MW\geq 0$ such that for all $x,y \in \rset^d$\alain{I think we can relax this condition, to discuss}
%  \begin{equation*}
  %  \ps{W(x) - W(y)}{x-y} \geq \mW \norm[2]{x-y} -\MW \eqsp.
  %\end{equation*}
  \end{enumerate}
\end{assumption}
Notice that by \eqref{flinear}, a sufficient condition for \Cref{assum:W}-\ref{assum:W_ii} is that $\nabla W$ is $L-$Lipschitz continuous, with $L< 2c/\varphi(R_0)$. By \Cref{assum:W}-\ref{assum:W_i}, $\nabla W(0) =0$. Thus, since $\fc^\prime\le 1$, \Cref{assum:W} implies 
\begin{equation}
\label{eq:bound_nabla_W}
  \norm{\nabla W( x) } \leq \eta \norm{x} \eqsp.
\end{equation}

We consider also the following additional condition:
\begin{assumption}
  \label{ass:confine_assum_convex_inf}
  There exists $ M_W \in [0,\infty )$ such that for all $x,y \in \rset^d$,
  \begin{equation*}
    \ps{\nabla W(x)- \nabla W(y)}{x-y} \geq  -M_W \eqsp.
  \end{equation*}
\end{assumption}
Under \Cref{assum:kappa} and \Cref{assum:W}, the equations \eqref{eq:non_linear_sde}
and \eqref{eq:sde_particle} both have unique strong solutions $(\barX_t)_{t
  \geq 0}$ and $\{(X_t^i)_{t \geq 0}, i=1,\ldots,N\}$ if we suppose sufficient integrability assumptions on the initial measures (say moments of order 4), see e.g.\ \cite[Theorem 2.6]{cgm-08}. In addition, 
\begin{equation}
  \label{eq:bound_moment_T_fixed_solution}
  \sup\nolimits_{t \in \ccint{0,T}} \expe{\normLigne[2]{\bar X_t} + \normLigne[2]{X_t^1}} < \infty\qquad\text{for all }T \geq 0 \eqsp.
\end{equation}

\begin{example}[Double-well potential]
A natural example which satisfies the assumptions above is the case of a double well confinement potential with quadratic interaction
 $$V(x)=\|x\|^4-a\|x\|^2,\qquad W(x)=\pm\|x\|^2,\qquad a>0 \eqsp,$$
 where the sign of the interaction implies attractiveness or repulsion.  
Using \cite[Lemma 1 and Example 4]{Eberle2015}, one has that the constant $c$ in \eqref{eq:iode_f_concave}, which is crucial for the rate of convergence to equilibrium, is of order $\Theta ( 1)$ for small $a$, whereas for large $a$, $\log (c^{-1})$ is of order $\Theta (a^2)$.
  \end{example}

\subsection{Main results}
Let $\wasserstein_{\fc}$ denote the Kantorovich ($L^1$ Wasserstein) distance on probability measures based on the underlying distance function $f(\| x-y\| )$ on $\mathbb R^d$, i.e.,
  \begin{equation*}
    \wasserstein_{\fc}(\nu,\mu) = \inf_{\xi \in \Pi(\nu,\mu)} \int_{\rset^d \times \rset^d} \fc[\norm{x-y}] \rmd \xi(x,y) \eqsp,
  \end{equation*}
  where $\Pi(\nu,\mu)$ is the set of couplings of $\mu$ and $\nu$, i.e., probability measures $\xi$ on
  $\borelSet(\rset^d \times \rset^d)$ such that for all $\eventA \in
  \borelSet(\rset^d)$, $\xi(\eventA \times \rset^d)=\mu(\eventA)$ and
 $\xi(\rset^d \times \eventA) = \nu(\eventA)$.  In the case where $\fc$ is the identity
  function $\fc(r) \equiv r$, $\mathcal W_f$ is the usual $L^1$
  Wasserstein distance $\wasserstein_{1}$. We also consider corresponding Kantorovich distances on probability measures $\hat\nu$, $\hat\mu$ on the configuration space $(\mathbb R^{d})^N$. Here we set
\begin{equation*}
    \wasserstein_{\ell^1(\fc )}(\hat\nu,\hat\mu) = \inf_{\xi \in \Pi(\hat\nu,\hat\mu)} \int_{\rset^{N\cdot d} \times \rset^{N\cdot d}} \frac 1N\sum_{i=1}^N\fc[\norm{x^i-y^i}] \rmd \xi(x,y) \eqsp.
  \end{equation*}  
For $t \geq 0$ and $\nu,\mu \in {\mathcal P}(\rset^d)$, we denote by $\barmu_{t}^{\nu}$ and
  $\mu_t^{\mu,N}$ the laws of $\barX_t$ and $X_t^{i,N}$, for an arbitrry $i$, with initial distributions $\nu$ and $\mu$ ($\mu$ being the common distribution of the i.i.d.\ random variables $X_0^{i,N}$, ${i \in \{1,\ldots,N\}}$. We now state our main result.
 
\begin{theorem}
  \label{theo:unif_prop}
  Assume \Cref{assum:kappa} and \Cref{assum:W}, and suppose that \Cref{ass:confine_assum_convex_inf} is satisfied or
  $\eta \in \ooint{0,m_V/2}$, where $m_V$ is given by \eqref{eq:v_convex_inf}. Let $\nu$ and $\mu$ be probability
  measures on $(\rset^d, \mathcal{B}(\rset^d))$ which admit a finite
  fourth moment. Then for all $t \geq 0$ and $N \in \nset$,
 we have
  \begin{equation*}
    \wasserstein_{\fc}\left(\mu_t^{\mu,N},\barmu_{t}^{\nu} \right)\ \leq\ 
 \rme^{-2(\c-\eta)t} \wasserstein_{\fc}\left(\nu , \mu \right)
+ (2(\c-\eta))^{-1}{ \ctheo \eta  N^{-1/2} }   \eqsp,
%    \rme^{-(c-\eta)t} \wasserstein_{\fc}(\nu,\mu) + (1-\rme^{-(c-\eta)t}) \eta\ctheo N^{-1/2} ,
  \end{equation*}
  and, more generally,
 \begin{equation*}
    \wasserstein_{\ell^1 (\fc )}\left( \mathcal L (X_t^{1,N},\ldots, X_t^{N,N}),
    (\barmu_{t}^{\nu})^{\otimes N}  \right)    
   \  \leq\ 
\rme^{-2(\c-\eta)t} \wasserstein_{\fc}(\nu , \mu )
+ (2(\c-\eta))^{-1}{ \ctheo \eta  N^{-1/2} }   \eqsp.
%    \rme^{-(c-\eta)t} \wasserstein_{\fc}(\nu,\mu) + (1-\rme^{-(c-\eta)t}) \eta\ctheo N^{-1/2} .
  \end{equation*} 
Here $\fc$ and $c$ are defined by \eqref{eq:def_f_concave} and \eqref{eq:def_cont_rate}, respectively, and 
$\ctheo$ is a constant that depends only on the dimension $d$, the second moment of $\nu$, as well as $V$ and $W$.
\end{theorem}
By \eqref{flinear} and the definition of $\wasserstein_{\f}$, we immediately obtain:
\begin{corollary}
  Under the same hypotheses as in \Cref{theo:unif_prop}, for all $t \geq 0$ and $N \in \nset$, 
  \begin{equation*}
    \wasserstein_{1}(\mu_t^{\mu,N},\barmu_{t}^{\nu} )\ \leq\    2\varphi (R_0)^{-1}\rme^{-2(\c-\eta)t} \wasserstein_{\fc}(\nu , \mu )
+ (\varphi (R_0)(\c-\eta))^{-1}{ \ctheo \eta  N^{-1/2} }   \eqsp,
  \end{equation*}
  where $\RZ$ and $\varphi$ are defined by \eqref{eq:def_RZ} and \eqref{eq:def_phi_concave}, respectively, and 
$\ctheo$ is a constant only depending on $d$, the second moment of $\nu$, as well as $V$ and $W$.
\end{corollary}
The dependence of the bounds on $N$ is of the right order and consistent with both the non uniform in time estimates, and the uniform in time bounds under uniform strict convexity for $L^2$ Wasserstein distances obtained for example in \cite{Szn91,malrieu03}.
% \begin{remark}
%   The proof of  $\eta \in \ooint{0,m_V/2}$ in \Cref{theo:unif_prop} is
%   not necessary in the case where $W$ is convex at infinity. 
% \end{remark}

\begin{remark}\label{rem:extensions}
For the sake of clarity  and comparison with other recent works on uniform propagation of chaos \cite{malrieu03,cgm-08}, we have restricted ourselves to interaction drifts of the type $b(x,\mu)=-\nabla W*\mu$. Similarly to \cite{egz-16}, the present result can be extended to $b(x,\mu)=\int b(x,y)\mu(dy)$ if $b$ satisfies a Lipschitz condition with sufficiently small Lip\-schitz constant. Similarly, the confinement drift $-\nabla V$ can be replaced by a non-gradient drift satisfying a corresponding assumption as \Cref{assum:kappa}.
\end{remark}

\begin{remark}
 In \cite{cmcv-03,malrieu03,cgm-08}, it has been shown that for a fixed center of mass, exponential contractivity of the nonlinear SDE and uniform propagation of chaos hold if 
$V+2W$ is, for example, strictly uniformly convex. This suggests that convexity 
of the interaction potential $W$ can make up for non-convexity of $V$, a fact that is not visible from our current approach. The problem is that
a symmetrization trick is used to get a benefit from the convexity of the interaction potential. This trick does not carry over in the same form to the $\ell_1$ type distances considered here. We will address this challenging issue in a future work.
  %For example if $W=W_1+W_2$ where $W_1$ satisfies \Cref{assum:kappa} and $W_2$ satisfies \Cref{assum:W}, we have to consider a distance $\wasserstein_{\tilde f}$ where $\tilde f(r)=\fc(r)$ when $r<R_3$ and $\tilde f(r)=r^2$ when $r>R_4$. However, this construction and the associated proofs are highly technical and we prefer  to stick to our setting for sake of clarity.
\end{remark}

\begin{remark}
By combining the result of Therorem \ref{theo:unif_prop} and the contraction result for mean-field particle systems in \cite[Corollary 3.4]{Eberle2015}, one can derive a contraction result for the non-linear equation \eqref{eq:non_linear_pde}, recovering essentially \cite[Theorem 2.3]{egz-16}.
\end{remark}

%Define the Wasserstein distance on $(\rset^d)^N$, for $N \in \nset$,
%$N \geq 1$, for all probability measure $\mu^N$ and $\nu^N$ on
%$\mathcal{B}((\rset^d)^N)$ by
%\begin{equation*}
%  \wasserstein_{l_1}(\mu^N,\nu^N) = \inf_{\xi \in \Pi(\mu^N,\nu^N)} \int_{(\rset^d)^N} \defEns{\sum_{i=1}^N \norm{x^i-y^i} } \rmd \xi(x,y) \eqsp. 
%\end{equation*}

%\textcolor{red}{The following corollary still has to be adapted. I think we can now simply combine the theorem above with our results on convergence to equilibrium for the nonlinear equation in the TAMS paper!}
%As another corollary, using the results of Eberle\cite[Corollary
%3.4]{Eberle2015}.
%
%\begin{corollary}
%Assume \Cref{assum:kappa} and \Cref{assum:W}.  Let $\mu$ be a probability
%  measure on $(\rset^d, \mathcal{B}(\rset^d))$ which admits a finite
%  fourth moment. Then there exists $\eta_0>0$ such that there exists $\kappa >0$ (independent of $N$) and $a >0$, for all $\eta \in \ooint{0,\eta_0}$,
%  $$\wasserstein_{l_1}\left(\lawX{(X_t^{i,N})}[i \in \{1,\ldots,N\}],\mu_\infty^N\right)\le a\,e^{-\kappa t}\wasserstein_{l_1}(\mu^{\otimes N},\mu_\infty^N) \eqsp,$$
%  where $\mu_{\infty}^N$ is the unique invariant measure of $\{(X_t^{i,N})_{i \in \{1,\ldots,N\}} \, : \, t \geq 0\}$
%\end{corollary}

\section{Proofs}\label{sec:proofs}
The proof of the main result is based on three ingredients. As usual, the starting construction is to consider i.i.d.\ copies of the nonlinear process and to couple them with the system of particles. This will be done by considering a reflection coupling coordinate by coordinate. The second ingredient will then be to consider a specific Kantorovich ($L^1$ Wasserstein) distance based on an $\ell_1$ type metric on the product space that is particularly suited to this coupling. The last ingredient is a law of large numbers type control in $L^2$. Let us first define the reflection coupling between the independent nonlinear processes and the mean-field particle system.

\subsection{Coupling by reflection}\label{sec:reflectioncoupling}
For every $\delta >0$, we consider Lipschitz continuous functions $\phis^{\delta},\phir^{\delta}: \rset^d \to \rset$ satisfying
\begin{equation}
\label{eq:phirs}
(\phis^{\delta})^{2}(x) + (\phir^{\delta})^{2} (x) = 1 \eqsp \text{ for all $x \in \rset^d$},\qquad
\phir^{\delta} (x) =\begin{cases} 1 & \text{if } \norm{x} \geq \delta \eqsp, \\ 
 0 \eqsp &\text{if } \norm{x} \leq \delta/2 \eqsp.\end{cases}
\end{equation}
%For example we can consider $\phis^{\delta}$ given for all $x \in \rset^d$, $\phis^{\delta}(x) = 1$ if $\norm{x} \leq \delta/2$, $\phis^{\delta}(x)  = \sin^2\{ 2\uppi (\delta^2 -\norm[2]{x})/(3\delta) \}$ for $\norm{x} \in \ccint{\delta/2,\delta}$ and $\phis^{\delta}(x) = 0$ for $\norm{x} \geq \delta$. 
Now fix $\delta >0$, probability measures $\nu ,\mu$ on $\mathbb R^d$ with finite fourth moment, and a coupling $\xi\in \Pi (\nu ,\mu )$. We consider the coupling between the independent nonlinear processes and the mean-field particle system defined by the following system of stochastic differential equations:
\begin{eqnarray*}
\rmd \barX_s^{i} &=& -\nabla V(\barX_s^{i}) \rmd s - \nabla W \ast \barmu_{s}^{\nu} (\barX_s^{i}) \rmd s  + \sqrt{2} \defEns{ \phir^{\delta}(\diffX^{i}_t) \rmd B_s^{i} + \phis^{\delta}(\diffX^{i}_t) \rmd \tildeB_s^{i}} ,\\
\rmd X_s^{i,N} &=& -\nabla V(X_s^{i,N}) \rmd s - \frac 1N \sum\nolimits_{j=1}^N \nabla W (X_s^{i,N}-X_s^{j,N}) \rmd s  \\
&& \phantom{-\nabla V(X_s^{i,N}) \rmd s - N^{-1} }
+ \sqrt{2} \defEns{ \phir^{\delta}(\diffX^{i}_t)\parenthese{\Id-2 \diffXn^{i}_t (\diffXn^{i}_t)^{\transpose} } \rmd B_s^{i} + \phis^{\delta}(\diffX^{i}_t) \rmd \tildeB_s^{i}} \eqsp.
\end{eqnarray*}
%  \label{eq:coupling_non_linear_particle}
Here we assume that
$(\bar X_0^{i},X_0^{i,N})$, $i \in \{1,\ldots , N\}$, are independent random variables with law $\xi$, $(B^i_s)_{s \geq 0}$ and $(\tildeB_s^i)_{s \geq 0}$, $i \in \{1,\ldots , N\}$,
are independent standard Brownian motions in $\mathbb R^d$ that are also independent of the initial conditions,
and
\begin{equation}
  \label{eq:diffX}
  \diffX_t^{i} = \barX_t^{i} - X_t^{i,N} \eqsp, \quad \diffXn_t^{i} = \mathrm{n}(\diffX_t^{i}) \eqsp, 
\end{equation}
where $\mathrm{n} : \rset^d \to \rset^d$ is given for all $x \in \rset^d\setminus\{ 0\}$ by $\mathrm{n}(x) = x / \norm{x}$, and $\mathrm{n}(0) = 0$.\smallskip 

Under the fourth moment assumption on $\nu$ and $\mu$, the system of SDEs has a unique strong solution. This can be shown similarly to \cite[Theorem 2.6]{cgm-08}.
By \eqref{eq:phirs} and L\'evy's characterization of Brownian motion, the process 
$((\barX_t^{1},\ldots ,\barX_t^{N} )_{t\ge 0}, (X_t^{1,N},\ldots ,X_t^{N,N})_{t\ge 0})$ is indeed a realization of a coupling between the system of independent nonlinear diffusions with initial law $\nu^{\otimes N}$ and the mean-field particle system with
initial law $\mu^{\otimes N}$. In particular,
\begin{equation}\label{eq:couplinglaws}
 \bar\mu_t^\nu=  \lawX{\bar X^i}[t]\quad\text{ and }\quad\mu_t^{\mu ,N}=  \lawX{X^{i,N}}[t]\qquad\text{for all }t\ge 0\text{ and }i \in \{1,\ldots , N\} \eqsp.
\end{equation}

\begin{lemma}\label{lem:mom}
  \label{lem:reg_O}
  Assume \Cref{assum:kappa} and \Cref{assum:W}. Then almost surely, for all $t \geq 0$ and $i \in \{1,\ldots,N\}$, 
$$
\rmd   \norm{\diffX^{i}_t} \ =\ - \ps{\nabla V(\barX_s^{i})-\nabla V(X_s^{i,N})}{\diffXn^{i}_t}\, \rmd t \, +\, A_t^i\, \rmd t\,  +\, 2 \sqrt{2}  \phir^{\delta}(\diffX^{i}_t) (\diffXn^{i}_t)^{\transpose}  \rmd B_t^{i} \eqsp,$$
where $(A_t^i)_{t\ge 0}$ is an adapted stochastic process such that
\begin{equation}\label{eq:Ati}
A^i_t \ \le \
\left\|{\nabla W \ast \bar\mu_t^\nu (\barX_t^{i}) -  \frac 1N \sum\nolimits_{j=1}^N \nabla W (X_t^{i,N}-X_t^{j,N})}\right\| \eqsp.
\end{equation}
%where $\diffX^{i}_t$ is defined by \eqref{eq:diffX} for all  and $t \geq 0$.
\end{lemma}

%\textcolor{red}{Statement and proof substantially revised. Please check again carefully.}

\begin{proof}
For all $i \in \defEns{1,\ldots,N}$, using Itô's formula, we have 
\begin{multline*}
\rmd   \norm[2]{\diffX^{i}_t} = - 2 \ps{\nabla V(\barX_t^{i})-\nabla V(X_t^{i,N})}{\diffX^{i}_t} \rmd t \\
 - 2\ps{\nabla W \ast \bar\mu_t^\nu (\barX_t^{i}) -  N^{-1} \sum\nolimits_{j=1}^N \nabla W (X_t^{i,N}-X_t^{j,N})}{\diffX^{i}_t} \rmd t\\
  + 4 \sqrt{2}  \phir^{\delta}(\diffX^{i}_t) \ps{ \diffXn^{i}_t}{\diffX^{i}_t} (\diffXn^{i}_t)^{\transpose}  \rmd B_t^{i} + 8  (\phir^{\delta}(\diffX^{i}_t))^2 \rmd t \eqsp.
\end{multline*}
Now let $a>0$, and consider the function $\psi_a(r) = (r+a)^{1/2}$ for
$r \geq 0$. Note that $\psi_a$ is infinitely
continuously differentiable on $(0,\infty)$, and for all $t \geq 0$,
$\lim_{a \to 0} \psi_a(t) = t^{1/2}$. Therefore, using again Itô's
formula, we get
\begin{align}
%\label{eq:reg_O_1}
\nonumber
\rmd   \psi_a\parenthese{\norm[2]{\diffX^{i}_t}} &= - 2 \psi_a'\parenthese{\norm[2]{\diffX^{i}_t}} \ps{\nabla V(\barX_t^{i})-\nabla V(X_t^{i,N})}{\diffX^{i}_t} \rmd t \\
\label{eq:reg_O_2}
& - 2\psi_a'\parenthese{\norm[2]{\diffX^{i}_t}} \ps{\nabla W \ast \bar\mu_t^\nu (\barX_t^{i}) -  N^{-1} \sum\nolimits_{j=1}^N \nabla W (X_t^{i,N}-X_t^{j,N})}{\diffX^{i}_t} \rmd t\\
\label{eq:reg_O_4}
 &+  (\phir^{\delta}(\diffX^{i}_t))^2 \defEns{8 \psi_a'\parenthese{\norm[2]{\diffX^{i}_t}} +16 \norm[2]{\diffX^{i}_t}\psi_a''\parenthese{\norm[2]{\diffX^{i}_t}}} \rmd t\\
\label{eq:reg_O_3}
 & + 4 \sqrt{2} \psi_a' \parenthese{\norm[2]{\diffX^{i}_t}}  \phir^{\delta}(\diffX^{i}_t) \ps{ \diffXn^{i}_t}{\diffX^{i}_t} (\diffXn^{i}_t)^{\transpose}  \rmd B_t^{i} \eqsp. 
\end{align}
Note that $2r\psi_a'(r^2)={r}/{\sqrt{r^2+a}}\le 1/2$ for all $a,r > 0$. In particular,  
\begin{equation*}
\left|  2\psi_a'\parenthese{\norm[2]{\diffX^{i}_t}} \ps{\nabla V(\barX_t^{i})-\nabla V(X_t^{i,N})}{\diffX^{i}_t} \right|
\ \leq\ \norm{\nabla V(\barX_t^{i})-\nabla V(X_t^{i,N})} \eqsp.
\end{equation*}
Therefore, by dominated convergence, we see that that almost surely for all $T \geq 0$,
%%
% $\psi_a'\parenthese{\norm[2]{\diffX^{i}_t}} \ps{\nabla V(\barX_s^{i})-\nabla V(X_s^{i,N})}{\diffX^{i}_t} = \1_{\norm{\diffX^{i}_t >0}} \psi_a'\parenthese{\norm[2]{\diffX^{i}_t}} \ps{\nabla V(\barX_s^{i})-\nabla V(X_s^{i,N})}{\diffX^{i}_t}$
$$
  \lim_{a \to 0} \int_{0}^T  2\psi_a'\parenthese{\norm[2]{\diffX^{i}_t}} \ps{\nabla V(\barX_t^{i})-\nabla V(X_t^{i,N})}{\diffX^{i}_t} \rmd t 
=  \int_{0}^T  \ps{\nabla V(\barX_t^{i})-\nabla V(X_t^{i,N})}{\diffXn^{i}_t} \rmd t \eqsp.
$$
Furthermore, for any $a>0$, the term in \eqref{eq:reg_O_2} is bounded from above by the expression on the right hand side of \eqref{eq:Ati}. Moreover, noting that 
$\phir^{\delta}(z) = 0$ for $z \in \rset^d$ with $\norm{z} \leq \delta/2$, and $8\psi_a'(r^2)+16r^2\psi_a''(r^2)=4a/(r^2+a)^{3/2}\le 4a/r^3$ for $a,r>0$, we see 
that for $T\ge 0$,
$$
  \lim_{a \to 0} \int_{0}^T   (\phir^{\delta}(\diffX^{i}_t))^2 \defEns{8 \psi_a'\parenthese{\norm[2]{\diffX^{i}_t}} +16 \norm[2]{\diffX^{i}_t}\psi_a''\parenthese{\norm[2]{\diffX^{i}_t}}} \rmd t  
\ =\  0 \eqsp.
$$
Finally, by \cite[Theorem 2.12]{revuz:yoz:1999}, we have almost surely
$$
  \lim_{a \to 0} \int_0^T4 \sqrt{2} \psi_a' \parenthese{\norm[2]{\diffX^{i}_t}}  \phir^{\delta}(\diffX^{i}_t) \ps{ \diffXn^{i}_t}{\diffX^{i}_t} (\diffXn^{i}_t)^{\transpose}  \rmd B_t^{i}
\ =\  \int_0^T 2 \sqrt{2}  \phir^{\delta}(\diffX^{i}_t) (\diffXn^{i}_t)^{\transpose}  \rmd B_t^{i} \eqsp,
$$
for any $T\ge 0$. This completes the proof of the lemma.
\end{proof}

\subsection{A moment control}
The following uniform moment bound will be important for the proof of our main result.

\begin{lemma}\label{lem:momentbound}
  Assume \Cref{assum:kappa} and \Cref{assum:W}, and suppose that $\eta\in \ooint{0,m_V/2}$ or Assumption \Cref{ass:confine_assum_convex_inf} is satisfied. 
%  Let $\nu$ be a
%  probability measure on $(\rset^d,\mathcal{B}(\rset^d))$ which has
%  a finite fourth moment and 
 Let $(\bar X_t)_{t \geq 0}$ be a solution of
  \eqref{eq:non_linear_sde} with  $\expe{\|\bar X_0\|^2}<\infty$.
%  $\bar X_0$ distributed according to
%  $\nu$. 
  Then there exists $C\in (0,\infty )$ depending only on $d,V,W$ and the second moment of $\bar X_0$ such that
  $$\sup\nolimits_{t\ge 0} \expe{\|\bar X_t\|^2}\le C \eqsp.$$
\end{lemma}

%\textcolor{red}{I could not see why we have to assume a finite fourth moment, so I deleted this assumption. Please check.}

\begin{proof}
The proof is quite standard but we include it here for completeness. First, by It\^o's formula, we have
$$(1/2) \rmd \|\bar X_t\|^2=-\langle \bar X_t,\nabla V(\bar X_t)\rangle \rmd t-\langle \bar X_t,\nabla W*\bar\mu^\nu_t(\bar X_s)\rangle \rmd t +d\, \rmd t+ \sqrt{2}\bar X_t^{\transpose} \rmd B_t \eqsp , $$
where $\nu$ denotes the law of $\bar X_0$. Let $(\tilde{X}_t)_{ t \geq 0} $ be an independent copy of $(\bar X_t)_{t \geq 0}$. Then by symmetrization and \eqref{eq:bound_moment_T_fixed_solution}, we get taking the expectation that
$$\frac{\rmd}{\rmd t} \expe{\|\bar X_t\|^2}\ =\ -2 \expe{\langle \bar X_t,\nabla V(\bar X_t)\rangle} -\expe{\langle \bar X_t- \tilde X_t,\nabla W(\bar X_t - \tilde X_t)\rangle } +2 d \eqsp.$$
Now suppose first that \Cref{assum:kappa}, \Cref{assum:W}
and \Cref{ass:confine_assum_convex_inf} are satisfied. Then
 by \eqref{eq:v_convex_inf} and since $\nabla W(0)=0$, 
  \begin{equation*}
   \frac{\rmd}{\rmd t} \expe{\|\bar X_t\|^2}  \leq 2M_V+M_W+2d-2 m_V \expe{\|\bar X_t\|^2} +2 \norm{\nabla V(0)}\expe{ \|\bar X_t\|^2}^{1/2}   \eqsp.
  \end{equation*}
Hence Gronwall's lemma concludes the proof.\smallskip

Alternatively, assume only that \Cref{assum:kappa} and \Cref{assum:W} are satisfied. Since $f(r)\le r$, we obtain
 \begin{equation*}
   \frac{\rmd}{\rmd t} \expe{\|\bar X_t\|^2}  \leq 2M_V+2d+2(\eta- m_V) \expe{\|\bar X_t\|^2} +2 \norm{\nabla V(0)}\expe{ \|\bar X_t\|^2}^{1/2}   \eqsp.
  \end{equation*}
  Hence for $\eta < m_V/2$, we can still apply Gronwall's lemma to conclude the proof.
\end{proof}

\subsection{Proof of main results}
We can now prove our main result.

%\textcolor{red}{Proof substantially revised. Please check.}

\begin{proof}[Proof of \Cref{theo:unif_prop}]
%Let $f: \rset_+ \to \rset_+$ be an increasing concave function continuously differentiable such that $f'$ is absolutely differentiable. 
We fix $\delta >0$ and a coupling $\xi\in\Pi (\nu ,\mu )$, and we consider the
reflection coupling with initial law $\xi^{\otimes N}$ between the system of nonlinear processes and the mean-field particle system as introduced in Subsection \ref{sec:reflectioncoupling}.
Since $\fc$ is continuously differentiable and $\fcp$ is absolutely continuous, by \Cref{lem:reg_O} and the Itô-Tanaka formula, we obtain
$$
\rmd \fc[\norm{\diffX^{i}_t}] \, =\,  (C_t^i+A_t^i)\, \fcp[\norm{\diffX^{i}_t}]\rmd t + 4 \fcpp[\norm{\diffX^{i}_t}] (\phir^{\delta}(\diffX^{i}_t))^{2} \rmd t  +  \sqrt{8}  \phir^{\delta}(\diffX^{i}_t)  (\diffXn^{i}_t)^{\transpose}  \rmd B_t^{i}\eqsp,$$
{where} $C_t^i = - \ps{\nabla V(\barX_t^{i})-\nabla V(X_t^{i,N})}{\diffXn^{i}_t}$.
Define $\omega : \rset_+ \to \rset_+$ by  
\begin{equation}
  \label{eq:def_omega}
  \omega(r) = \sup\nolimits_{s \in \ccint{0,r} } s \kappa(s)^- \eqsp.
\end{equation}
By definition of $\kappa$ given by \eqref{eq:def_kappa}, and 
by \eqref{eq:iode_f_concave}, 
\begin{eqnarray}\nonumber
\lefteqn{ C_t^i\, \fcp[\norm{\diffX^{i}_t}]\, +\, 4 \fcpp[\norm{\diffX^{i}_t}] (\phir^{\delta}(\diffX^{i}_t))^{2} }\\
 &\le & 
-\norm{\diffX^{i}_t} \kappa(\norm{\diffX^{i}_t})  
 \fcp[\norm{\diffX^{i}_t}]
\, +\, 4 \fcpp[\norm{\diffX^{i}_t}] (\phir^{\delta}(\diffX^{i}_t))^{2}\label{eq:Cbound} \\
&\le &-2c \fc[\norm{\diffX^{i}_t}](\phir^{\delta}(\diffX^{i}_t))^{2}  + \omega(\delta) 
 \\& \le&
-2c \fc[\norm{\diffX^{i}_t}]  + \omega(\delta)  + 2 c \fc[\delta]
\, .\nonumber
\end{eqnarray}
Moreover, by \eqref{eq:Ati}, we can estimate $A_t^i \, \le\, N^{-1} \sum\nolimits_{j=1}^N \Xi_t^{i,j} + \Upsilon_t^{i}$, \ where
%\begin{multline*}
%\rmd  \fc[\norm{\diffX^{i}_t}] \leq 2 \sqrt{2}  \phir^{\delta}(\diffX^{i}_t)  (\diffXn^{i}_t)^{\transpose}  \rmd B_t^{i} \\ +\parentheseDeux{4 \fcpp[\norm{\diffX^{i}_t}] (\phir^{\delta}(\diffX^{i}_t))^{2} 
%-\defEns{\norm{\diffX^{i}_t} \kappa(\norm{\diffX^{i}_t})  
%%-\ps{  N^{-1} \sum_{j=1}^N \nabla W (\barX_t^{i,N}-\barX_t^{j,N}) -\nabla W (X_t^{i,N}-X_t^{j,N})}{\diffXn^{i}_t}
%+N^{-1} \sum_{j=1}^N \Xi_t^{i,j} + \Upsilon_t^{i} } \fcp[\norm{\diffX^{i}_t}]}\rmd t 
%% \defEns{
%% - \ps{\nabla W \ast u_t (\barX_t^{i}) -  N^{-1} \sum_{j=1}^N \nabla W (\barX_t^{i}-\barX_t^{j})}{\diffXn^{i}_t}
%% } \fcp[\norm{\diffX^{i}_t}] \rmd t \eqsp ,
%\end{multline*}
\begin{align*}
\Xi_{t}^{i,j} &\ =\  \norm{ \nabla W (\barX_t^{i,N}-\barX_t^{j,N}) -\nabla W (X_t^{i,N}-X_t^{j,N})},\\
\Upsilon^i_t&\ =\ \norm{\nabla W \ast \bar\mu_t^\nu (\barX_t^{i}) -  N^{-1} \sum\nolimits_{j=1}^N \nabla W (\barX_t^{i}-\barX_t^{j})} \eqsp.
\end{align*}
%Define $\omega : \rset_+ \to \rset_+$ for all $r \in \rset_+$ by  
%\begin{equation}
%  \label{eq:def_omega}
%  \omega(r) = \sup\nolimits_{s \in \ccint{0,r} } s \kappa(s)^- \eqsp.
%\end{equation}
%Using  that for all $t > 0$, $\proba{\norm{\diffX^{i}_t} = \RO} = 0$ and 
%\eqref{eq:iode_f_concave}, we get almost surely
%\begin{align}
%\nonumber
%\rmd  \fc[\norm{\diffX^{i}_t}] &\leq 2 \sqrt{2}  \phir^{\delta}(\diffX^{i}_t)  (\diffXn^{i}_t)^{\transpose}  \rmd B_s^{i} \\
%\nonumber 
%& +\parentheseDeux{-2c \fc[\norm{\diffX^{i}_t}](\phir^{\delta}(\diffX^{i}_t))^{2}  + \omega(\delta) 
%-\defEns{N^{-1} \sum_{j=1}^N \Xi_t^{i,j} + \Upsilon_t^{i} } \fcp[\norm{\diffX^{i}_t}] } \rmd t \\
%\nonumber
%& \leq 2 \sqrt{2}  \phir^{\delta}(\diffX^{i}_t)  (\diffXn^{i}_t)^{\transpose}  \rmd B_s^{i} \\
%\nonumber
%& +\parentheseDeux{-2c \fc[\norm{\diffX^{i}_t}]  + \omega(\delta)  + 2 c \fc[\delta]
%-\defEns{N^{-1} \sum_{j=1}^N \Xi_t^{i,j} + \Upsilon_t^{i} } \fcp[\norm{\diffX^{i}_t}]} \rmd t \eqsp.
%\end{align}
Therefore, and since $(\int_{0}^t \phir^{\delta}(\diffX^{i}_s)  (\diffXn^{i}_s)^{\transpose} \rmd
B_s^{i})_{t \geq 0}$ is a martingale and $f'\le 1$, we obtain
\begin{eqnarray}
\nonumber
 \frac 1N\sum_{i=1}^N \frac{\rmd}{\rmd t}\expe{\fc[\norm{\diffX^{i}_t}]}  & \leq & - \frac {2c}N\sum_{i=1}^N\expe{ \fc[\norm{\diffX^{i}_t}]}\, +\, \frac 1N \sum_{i=1}^N  \expe{\frac 1N \sum_{j=1}^N \Xi_t^{i,j} + \Upsilon_t^{i} }  \\   &&\qquad + {\omega(\delta)  + 2 c \fc[\delta]} 
\label{eq:1}
\end{eqnarray}
for a.e.\ $t\ge 0$. By Assumption \Cref{assum:W}, for all $i,j \in \{1,\ldots,N\}$ and $t \geq 0$, 
\begin{equation}\label{eq:Xibound}
{\Xi^{i,j}_t} \
 \leq\ \eta\,  \fc[\norm{\diffX^{i}_t}+\norm{\diffX^{j}_t}]\ \le\ \eta\, \left(\fc [\norm{\diffX^{i}_t}]+\fc[\normLigne{\diffX^{j}_t}]\right) \eqsp.
\end{equation}
Moreover, in order to control $\Upsilon_t^i$, we remark that given $\bar X_t^i$, the random variables $\barX_t^{j}$, $j\neq i$, are i.i.d.\ with law
$\bar\mu_t^\nu$. In particular,
$$\expe{\nabla W (\barX_t^{i}-\barX_t^{j})|\bar X_t^i}\ =\ \nabla W \ast \bar\mu_t^\nu (\barX_t^{i})\eqsp.$$
Since by \Cref{assum:W}, $\nabla W(0)=0$ and $\nabla W$ is Lipschitz with
constant $\eta$, we obtain 
\begin{eqnarray*}
\lefteqn{\expe{\left.\ \norm{\nabla W \ast \bar\mu_t^\nu (\barX_t^{i})-\frac{1}{N-1}\sum\nolimits_{j=1}^N\nabla W (\barX_t^{i}-\barX_t^{j})}^2\right|\, \bar X_t^i}}\\
&=&\frac{1}{N-1}\text{Var}_{\bar\mu_t^\nu}\left( \nabla W (\barX_t^{i}-\cdot)\right)\
\le\ \frac{\eta^2}{N-1}\int \norm{x}^2\, \bar\mu_t^\nu(\rmd x)\eqsp.
\end{eqnarray*}
Hence, by the Cauchy-Schwarz inequality and \eqref{eq:bound_nabla_W},
\begin{eqnarray*}
\expe{{\Upsilon_t^i} } &\le & \expe{\ \norm{\nabla W \ast \bar\mu_t^\nu (\barX_t^{i})-\frac{1}{N-1}\sum\nolimits_{j=1}^N\nabla W (\barX_t^{i}-\barX_t^{j})}}\\
&&+\left(\frac{1}{N-1}-\frac 1N\right)\sum\nolimits_{j=1}^N \expe{\eta\, \norm{\barX_t^{i}-\barX_t^{j}}}\\
&\le &\eta\, \left(\frac{1}{\sqrt{N-1}}+\frac{\sqrt 2}{N}\right)\, \left(\int \norm{x}^2\, \bar\mu_t^\nu(\rmd x)\right)^{1/2} \eqsp.
\end{eqnarray*}
By Lemma \ref{lem:momentbound}, we conclude that there is an explicit finite constant $\ctheo$ such that for $N\ge 2$,
\begin{equation}\label{eq:Upsilonbound}
\sup\nolimits_{t \geq 0} \expe{{\Upsilon_t^i} }
\
\leq\   \ctheo \eta N^{-1/2}  \eqsp,\qquad i=1,\ldots ,N \eqsp.
\end{equation}
Now, combining \eqref{eq:1}, \eqref{eq:Xibound} and \eqref{eq:Upsilonbound}, we finally obtain
%
%where $\ctheo$ may be explicitly derived.
%Plugging this bound in \eqref{eq:1}, we get 
%\begin{align}
%%\label{eq:1}
%\nonumber
%  N^{-1}\sum_{i=1}^N \expe{\fc[\norm{\diffX^{i}_t}]} &\leq -2c  N^{-1}\sum_{i=1}^N \int_{0}^t \expe{ \fc[\norm{\diffX^{i}_s}]} \rmd s  + \parenthese{\omega(\delta)  + 2 c \fc[\delta] + \ctheo N^{-1/2} }t \\
%\label{eq:1}
%&\qquad \qquad - N^{-1} \sum_{i=1}^N \int_0^t \expe{N^{-1} \sum_{j=1}^N \Xi_t^{i,j}  } \fcp[\norm{\diffX^{i}_t}] \rmd t \eqsp.
%\end{align}
%
%%\alain{check typos $s$ instead $t$}
%Using this result in  \eqref{eq:1}, it yields
\begin{eqnarray*}
%\label{eq:1}
 \frac{\rmd}{\rmd t}\, \frac 1N\sum\limits_{i=1}^N \expe{\fc[\norm{\diffX^{i}_t}]}
&  \leq &  -2\frac{\c-\eta }{N}\sum\limits_{i=1}^N \expe{ \fc[\norm{\diffX^{i}_t}]} 
  + {\omega(\delta) + 2 c \fc[\delta] + \frac{\ctheo \eta}{\sqrt N}   }
\end{eqnarray*}
for a.e.\ $t\ge 0$. Assuming $\eta < c$, we can conclude that
\begin{eqnarray*}
  \frac 1N\sum\nolimits_{i=1}^N \expe{\fc[\norm{\diffX^{i}_t}]} & \leq & \rme^{-2(\c-\eta)t}\frac 1N\sum\nolimits_{i=1}^N \expe{\fc[\norm{\diffX^{i}_0}]} \\
&&\qquad + (2(\c-\eta))^{-1}\parenthese{\omega(\delta)  + 2 c \fc[\delta] + \ctheo \eta  N^{-1/2} }.
\end{eqnarray*}
Noting that $\expe{\fc[\norm{\diffX^{i}_0}]}\le\int \|x-y\|\,d\xi (x, y)$ for all $i\in\{ 1,\ldots d\}$, we obtain
\begin{eqnarray*}
\lefteqn{ \wasserstein_{\ell^1 (\fc )}\left( \mathcal L (X_t^{1,N},\ldots, X_t^{N,N}),
    (\barmu_{t}^{\nu})^N  \right)}\\ 
    & \leq & \rme^{-2(\c-\eta)t}\int \|x-y\|\,d\xi (x, y)
+ (2(\c-\eta))^{-1}\parenthese{\omega(\delta)  + 2 c \fc[\delta] + \ctheo \eta  N^{-1/2} } \eqsp.
\end{eqnarray*}
By \eqref{eq:def_omega} and \Cref{assum:kappa}, $\lim_{\delta \to 0^+} \omega(\delta) = 0$. Hence taking the limit as $\delta$ goes to $0$, and the infimum over all couplings $\xi\in \Pi (\nu ,\mu )$ concludes the proof of the second inequality in \Cref{theo:unif_prop}. The first inequality follows similarly, noting that
by \eqref{eq:couplinglaws}, for all $i\in\{ 1,\ldots d\}$ and $t\ge 0$,
$$\mathcal{W}_f(\bar\mu_t^\nu,\mu_t^{\mu ,N})\ \le\ \expe{\fc[\norm{\bar X^{i}_t- X^{i,N}_t}]}\ =\ \expe{\fc[\norm{\diffX^{i}_t}]} \eqsp.$$
\end{proof}

\phantom{b}

%%% Local Variables: 
%%% mode: latex
%%% TeX-master: "main"
%%% End: 

{\bf Acknowledgment.} This work was initiated through a Procope project which is greatly acknowledged. The work was partially supported by ANR-17-CE40-0030, by the Hausdorff Center for Mathematics, and by the grant 346300 for IMPAN from the Simons Foundation and the matching 2015-2019 Polish MNiSW fund.

\nocite{*}
\bibliographystyle{plain}
\bibliography{bibliography}

\end{document}